\documentclass[11pt,reqno]{amsart}

\usepackage{amssymb,a4wide,amsthm,amsmath,amsfonts}
\usepackage{thmtools}
\usepackage{mathrsfs}
\usepackage{verbatim}
\usepackage{color}
\usepackage{enumitem}
\usepackage{bm}
\usepackage[parfill]{parskip}
\usepackage{graphicx}

\def\mathcal{\mathscr}

\usepackage{catoptions}
\makeatletter

\def\Autoref#1{%
  \begingroup
  \edef\reserved@a{\cpttrimspaces{#1}}%
  \ifcsndefTF{r@#1}{%
    \xaftercsname{\expandafter\testreftype\@fourthoffive}
      {r@\reserved@a}.\\{#1}%
  }{%
    \ref{#1}%
  }%
  \endgroup
}
\def\testreftype#1.#2\\#3{%
  \ifcsndefTF{#1autorefname}{%
    \def\reserved@a##1##2\@nil{%
      \uppercase{\def\ref@name{##1}}%
      \csn@edef{#1autorefname}{\ref@name##2}%
      \autoref{#3}%
    }%
    \reserved@a#1\@nil
  }{%
    \autoref{#3}%
  }%
}

\usepackage[colorlinks, linkcolor=blue]{hyperref}

\declaretheorem[style=plain,numberwithin=section]{theorem}
\declaretheorem[style=plain,sibling=theorem]{lemma}
\declaretheorem[style=plain,sibling=theorem]{proposition}

\declaretheorem[style=definition,sibling=theorem]{definition}
\declaretheorem[style=remark,sibling=theorem]{remark}
\declaretheorem[style=remark,sibling=theorem]{example}

\newcommand{\numberthis}{\addtocounter{equation}{1}\tag{\theequation}}
\numberwithin{equation}{section}

\renewcommand{\d}{\,\mathrm{d}}
\DeclareMathOperator*{\E}{\,\mathbb{E}}
\DeclareMathOperator*{\limweak}{\textit{w}^*-\lim\,}

\DeclareMathOperator*{\R}{{\mathbb{R}}}

\DeclareMathOperator*{\Dom}{\mathrm{Dom}}
\renewcommand{\L}{\mathcal{L}}
\renewcommand{\O}{\mathcal{O}}

\makeatletter
\def\namedlabel#1#2{\begingroup
    #2%
    \def\@currentlabel{#2}%
    \phantomsection\label{#1}\endgroup
}
\makeatother

\makeatletter
\DeclareRobustCommand{\cev}[1]{%
  \mathpalette\do@cev{#1}%
}
\newcommand{\do@cev}[2]{%
  \fix@cev{#1}{+}%
  \reflectbox{$\m@th#1\vec{\reflectbox{$\fix@cev{#1}{-}\m@th#1#2\fix@cev{#1}{+}$}}$}%
  \fix@cev{#1}{-}%
}
\newcommand{\fix@cev}[2]{%
  \ifx#1\displaystyle
    \mkern#23mu
  \else
    \ifx#1\textstyle
      \mkern#23mu
    \else
      \ifx#1\scriptstyle
        \mkern#22mu
      \else
        \mkern#22mu
      \fi
    \fi
  \fi
}
\makeatother

\begin{document}
\title[]
{Limiting measure and stationarity of solutions to stochastic evolution equations with Volterra noise}
\author{P. \v{C}oupek}
\address{Charles University\\
Faculty of Mathematics and Physics\\
Sokolovsk\'{a} 83\\
Prague 8\\
Czech Republic}
\email{coupek@karlin.mff.cuni.cz}

\thanks{The author was supported by the Charles University, project GAUK No. 322715, and by the Czech Science Foundation, project GA\v{C}R No. 15-08819S}


\begin{abstract}
Large-time behaviour of solutions to stochastic evolution equations driven by two-sided regular Volterra processes is studied. The solution is understood in the mild sense and takes values in a separable Hilbert space. Sufficient conditions for the existence of limiting measure and strict stationarity of the solution process are found and an example for which these conditions are also necessary is provided. The results are further applied to the heat equation perturbed by the two-sided Rosenblatt process.
\end{abstract}
\maketitle

\section{Introduction}
Consider the stochastic evolution equation 
	\begin{equation*}
		\left\{\begin{array}{cl}
					\d{X}_t & =AX_t\d{t} + \varPhi\d{B}_t, \quad t\geq 0,\\
					X_0 & = x,
				\end{array}\right.
	\end{equation*}
where $A$ generates a $C_0$-semigroup of bounded linear operators $S=(S(t), t\geq 0)$ acting on a separable Hilbert space and its mild solution which is defined by the variation of constants formula
	\begin{equation*}
		X^x_t := S(t)x + \int_0^tS(t-r)\varPhi\d{B}_r, \quad t\geq 0.
	\end{equation*}
The noise process is a two-sided Hilbert space valued $\alpha$-regular Volterra process $B$ (see \autoref{def:cylindrical_process}). It is shown (see \autoref{prop:limiting_measure}) that if the process $B$ has stationary and reflexive increments (see \autoref{def:stationarity_reflexivity}) and the equation satisfies certain stability conditions (see formula \eqref{eq:condition_on_inv_measure}), there is a limiting measure $\mu_\infty$ such that the law of $X^0_t$ and converges to $\mu_\infty$ as $t\rightarrow\infty$. Furthermore, we provide an example for which the stability condition is also a necessary one (see \autoref{ex:non_exp_stable}). Additionally, if the semigroup is strongly stable, we have (see \autoref{prop:limiting_measure_2}) that the law of $X_t^x$ tends to $\mu_\infty$ as $t\rightarrow\infty$ for each initial condition $x\in L^2(\Omega;V)$. Also, it is shown (see \autoref{prop:strict_stationarity}) that there exists an initial condition $x_\infty$, such that the solution $X^{x_\infty}$ is a strictly stationary process.

Volterra processes have been considered in the pioneering work \cite{AlosMazNua01} where the authors considered Gaussian Volterra processes (see also \cite{BauNua03,ErrEss09, Hida60}). Regular Volterra processes which might not be Gaussian and stochastic evolution equations driven by them were studied in the literature as well. In particular, existence and regularity results were given in \cite{BonTud11, CouMas16, CouMasOnd17, CouMasSnup17} and the present paper can be viewed as a continuation of the work. For specific cases of the driving noise, stationarity and large-time behavior of the solutions have been already treated in the literature (see e.g. \cite{DunMasDun02, MasNua03, MasPos07, MasPos08} and others for equations driven by the fractional Brownian motion (fBm)).

It is not a priori clear how the two-sided Volterra processes should be defined. We propose such a definition (\autoref{def:Volterra_process}) after analysis of two main examples of two-sided $\alpha$-regular Volterra processes - the fractional Brownian motion with Hurst parameter $H\in (1/2,1)$ (see e.g. \cite{AlosNua03, DecrUstu99, MandelbrotVanNess} for its definition and properties) and the Rosenblatt process (see e.g. \cite{Taqqu11, Tud08} for its definition and properties).

The paper is organized as follows.

In \autoref{sec:prelim}, we define two-sided $\alpha$-regular Volterra processes and give two examples - the (two-sided) fBm of Hurst parameter $H\in (1/2,1)$ and the (two-sided) Rosenblatt process. Then we modify the already existing stochastic integral with respect to one-sided $\alpha$-regular Volterra processes to the case when the integrator is two-sided and give basic properties of the integral. \Autoref{sec:limiting_measure} contains the main results of the paper. In particular, we find sufficient conditions for the existence of a limiting measure (\autoref{prop:limiting_measure} and \autoref{prop:limiting_measure_2}) and show the existence of an initial condition $x_\infty$ such that the mild solution $X^{x_\infty}$ is a stationary process (\autoref{prop:strict_stationarity}). The paper is concluded with two examples in \autoref{sec:examples}. The first is an example of a stochastic evolution equation for which the sufficient condition from \autoref{prop:limiting_measure} is also a necessary one and the second example is the stochastic heat equation driven by the two-sided space-time Rosenblatt process.

\section{Preliminaries}
\label{sec:prelim}

\subsection{Two-sided Volterra processes}
\label{sec:VP}
Let $K:{\R}^2\rightarrow\R$ be a kernel such that 
	\begin{itemize}
		\itemsep-0.5em
		\item $K(t,r)=0$ on $\{t<r\}$ and $\lim_{t\rightarrow r+} K(t,r) = 0$ for every $r\in\R$.
		\item $K(\cdot,r)$ si continuously differentiable in $(r,\infty)$ for every $r\in\R$. 
		\item There is an $\alpha\in (0,\frac{1}{2})$ such that 
			\begin{equation}
			\label{eq:kernel_estimate}
				\left|\frac{\partial K}{\partial u}(u,r)\right|\lesssim (u-r)^{\alpha-1}
			\end{equation}				
			 on $\{r<u\}$.
	\end{itemize}
Throughout the paper, $A\lesssim B$ means that there is a finite positive constant $C$ such that $A\leq CB$ uniformly. Such a function $K$ is called an \textit{$\alpha$-regular Volterra kernel} in the sequel (cf. \cite{CouMas16} and \cite{CouMasOnd17} where a slightly different estimate on the kernel is considered). Under these conditions, we can define
	\begin{equation}
	\label{eq:definition_of_R}
		R(s_1,t_1,s_2,t_2):= \int_{\R}\left(K(t_1,r)-K(s_1,r)\right)\left(K(t_2,r)-K(s_2,r)\right)\d{r}
	\end{equation}
which is finite for every $s_1,t_1,s_2,t_2\in\R$.

\begin{definition}
\label{def:Volterra_process}
A stochastic process $b=(b_t, t\in\R)$ is an \textit{$\alpha$-regular Volterra process} if it is centred, $b_0=0$, and such that 
	\begin{equation}
	\label{eq:generalized_VP_increments}
		\E(b_{t_1}-b_{s_1})(b_{t_2}-b_{s_2}) = R(s_1,t_1,s_2,t_2)
	\end{equation}
for every $s_1,s_2,t_2,t_2\in\R$, where $R$ is defined by \eqref{eq:definition_of_R} with an $\alpha$-regular Volterra kernel $K$. 
\end{definition}

\begin{remark}
Note that condition \eqref{eq:generalized_VP_increments} together with the properties of the kernel $K$ already imply that the process $b$ from \autoref{def:Volterra_process} has a version with $\varepsilon$-H\"{o}lder continous sample paths for every $\varepsilon\in (0,\alpha)$. This follows by using \eqref{eq:cov_increments} and (i) of \autoref{lem:inner_prod} below. In particular, for $t>s$, we obtain
	\begin{equation*}
		\E(b_t-b_s)^2 \lesssim \int_s^t\int_s^t|u-v|^{2\alpha-1}\d{u}\d{v} =\frac{1}{\alpha(1+2\alpha)} (t-s)^{1+2\alpha}
	\end{equation*}
and use the Kolmogorov continuity theorem. We always consider this continuous version.
\end{remark}

\begin{remark}
The condition \eqref{eq:generalized_VP_increments} is an analogue of the condition 
	\begin{equation*}
		{\E}b_tb_s = \int_0^{s\wedge t}K(t,r)K(s,r)\d{r}, \quad s,t\geq 0,
	\end{equation*}
for a one-sided Volterra process $b=(b_t,t\geq 0)$. See \cite{CouMas16, CouMasOnd17, CouMasSnup17} for the precise conditions on $K$ in the one-sided case.
\end{remark}

The existence of limiting measure and stationarity of solutions to stochastic evolution equations will be proved for equations driven by Volterra processes whose increments are stationary and reflexive. Let us state precisely what we mean by these two notions in the finite-dimensional case. 

\begin{definition}
\label{def:stationarity_reflexivity}
Let $d\geq 1$. We say that an ${\R}^d$-valued stochastic process $Y=(Y_t, t\in\R)$ has 
	\begin{itemize}
	\itemsep0em
		\item \textit{stationary increments} if for every $n\in\mathbb{N}$ and every $s_i, t_i\in\R$, $s_i<t_i$, $i=1,2,\ldots, n$, we have that the following holds for every $h\in\R$:
			\begin{equation}
			\label{eq:stat_increments}
			\begin{split}
				Law (Y_{t_1}-Y_{s_1}, Y_{t_2}-Y_{s_2}, \ldots, Y_{t_n}-Y_{s_n})  & = \\
						&  \hspace{-5.2cm}= Law(Y_{t_1+h}-Y_{s_1+h},Y_{t_2+h}-Y_{s_2+h}, \ldots, Y_{t_n+h}-Y_{s_n+h});
			\end{split}
			\end{equation}
		\item \textit{reflexive increments} if for every $n\in\mathbb{N}$ and every $s_i, t_i\in\R$, $s_i<t_i$, $i=1,2,\ldots, n$, we have that 
			\begin{align*}
		Law (Y_{t_1}-Y_{s_1}, Y_{t_2}-Y_{s_2}, \ldots, Y_{t_n}-Y_{s_n}) & = \\
						& \hspace{-4cm} = Law(Y_{-s_1}-Y_{-t_1},Y_{-s_2}-Y_{-t_2}, \ldots, Y_{-s_n}-Y_{-t_n}).
			\end{align*}
	\end{itemize}
\end{definition}

\begin{remark}
The above definition of stationary increments is stronger than strict stationarity of the increment process $(Y_{t+h}-Y_t, t\in\R)$ for every $h\geq 0$. In fact, in \autoref{def:stationarity_reflexivity}, we allow each increment to be of different length $h_i:=t_i-s_i$. This stronger concept is needed in \autoref{lem:law_of_integrals} where we cannot assume equidistant partitions while approximating a general integrand $f$.
\end{remark}

\begin{remark}
Note that the notion of stationary increments from \autoref{def:stationarity_reflexivity} does not imply reflexivity of the increments.
\end{remark}

We now give two examples of two-sided $\alpha$-regular Volterra processes with stationary and reflexive increments, namely, the (two-sided) fractional Brownian motion (fBm) and the (two-sided) Rosenblatt process.

\begin{example}
Recall the following representation of the two-sided fractional Brownian motion (see \cite{MandelbrotVanNess} or \cite{Taqqu11}):
	\begin{equation}
	\label{eq:fBm_def}
		W^H_t = C_H\int_{\R}\left((t-r)_{+}^{H-\frac{1}{2}}-(-r)_{+}^{H-\frac{1}{2}}\right)\d{W}_r
	\end{equation}
where $W=(W_t, t\in\R)$ is the two-sided standard Wiener process and $C_H$ is a normalizing constant such that ${\E}(W_1^H)^2=1$, i.e.
	\begin{equation*}
		C_H:= \sqrt{\frac{2H}{\left(H-\frac{1}{2}\right)\mathrm{B}\left(H-\frac{1}{2},2-2H\right)}}
	\end{equation*}
with $\mathrm{B}$ being the Beta function. Let us assume that $H\in (1/2,1)$. If we define 
	\begin{equation}
	\label{eq:fBm_kernel}
		K^H(t,r):=c_H\int_r^t(u-r)^{H-\frac{3}{2}}\d{u}, \quad -\infty<r<t,
	\end{equation}
with $c_H:=C_H(H-\frac{1}{2})$, then we have 
	\begin{equation*}
		W_t^H = \int_{\R}\left(K^H(t,r)-K^H(0,r)\right)\d{W}_r.
	\end{equation*}
Hence, the increments of the two-sided fBm of $H>\frac{1}{2}$ satisfy
	\begin{equation*}
		W_t^H-W_s^H = \int_{\R} (K^H(t,r)-K^H(s,r))\d{W}_r, \quad s,t\in\R.
	\end{equation*}
Two immediate facts follow from this representation. First, we see that the two-sided fBm of $H>\frac{1}{2}$ is in fact a Volterra process as defined in \autoref{def:Volterra_process}. Second, for $-\infty <s_1<t_1<\infty$ and $-\infty < s_2<t_2<\infty$, we have that
	\begin{equation}
	\label{eq:cov_increments}
		\E(W^H_{t_1}-W^H_{s_1})(W^H_{t_2}-W^H_{s_2}) = H(2H-1)\int_{s_1}^{t_1}\int_{s_2}^{t_2}|u-v|^{2H-2}\d{u}\d{v}.
	\end{equation}
The last equality gives
	\begin{equation*}
		\E(W_t^H-W_s^H)^2 = |t-s|^{2H}, \quad s,t\in\R,
	\end{equation*}
which allows us to recover the covariance function 
	\begin{equation}
	\label{eq:cov_fBm}
		\E W^H_sW^H_t = \frac{1}{2}\left(|s|^{2H}+|t|^{2H}-|t-s|^{2H}\right), \quad s,t\in\R.
	\end{equation}
It follows, moreover, that $W^H$ has stationary and reflexive increments. See also \cite{DecrUstu99} for its further properties.
\end{example}

\begin{remark}
Notice that the formula \eqref{eq:fBm_def} could be written (if the integrals converged) as
	\begin{equation*}
		W^H_t = C_H \int_{-\infty}^t(t-r)^{H-\frac{1}{2}}\d{W}_r - C_H\int_{-\infty}^0(-r)^{H-\frac{1}{2}}\d{W}_r.
	\end{equation*}
Thus, as suggested in \cite[Remark 3.4]{Hai05}, the process $W_t^H$ should rather be seen as a convergent difference of two divergent integrals $\tilde{W}^H_t-\tilde{W}^H_0$ where $\tilde{W}^H_t$ is given by
	\begin{equation*}
		\tilde{W}^H_t := C_H\int_{-\infty}^t(t-r)^{H-\frac{1}{2}}\d{W}_r = \int_{-\infty}^tK^H(t,r)\d{W}_r.
	\end{equation*}
\end{remark}

\begin{example}
\label{ex:Rosenblatt}
Similarly as in the case of the fBm above, we may also extend the Rosenblatt process to the whole real line. Recall the definition of the (one-sided) Rosenblatt process (see \cite{Taqqu11} or \cite{Tud08}). Let $H\in (1/2,1)$ and 
	\begin{equation*}
		R_t^H := A_H\int_{{\R}^2}^{\prime}\left(\int_0^t(u-y_1)_{+}^{-\frac{2-H}{2}}(u-y_2)_{+}^{-\frac{2-H}{2}}\d{u}\right)\d{W}_{y_1}\d{W}_{y_2}, \quad t\geq 0,
	\end{equation*}
where $A_H$ is a normalizing constant such that ${\E}(R_t^H)^2=1$, i.e.
	\begin{equation*}
		A_H := \frac{\sqrt{\frac{H}{2}(2H-1)}}{\mathrm{B}\left(\frac{H}{2},1-H\right)} =: \frac{\sigma}{\mathrm{B}\left(\frac{H}{2},1-H\right)}
	\end{equation*}
with $\mathrm{B}$ being the Beta function. The double integral is the Wiener-It\^{o} multiple integral of order $2$ with respect to the two-sided standard Wiener process $W=(W_t, t\in\R)$ where the prime means that the integration excludes the diagonal $y_1=y_2$ (see \cite{PecTaq11}). The inner integral can be written as the difference
	\begin{equation*}
		A_H\int_0^t(u-y_1)_{+}^{-\frac{2-H}{2}}(u-y_2)_{+}^{-\frac{2-H}{2}}\d{u} = K^H(t, y_1,y_2) - K^H(0, y_1, y_2)
	\end{equation*}
where 
	\begin{equation*}
		K^H(t, y_1,y_2) := A_H\int_{y_1\vee y_{2}}^t(u-y_1)^{-\frac{2-H}{2}}(u-y_2)^{-\frac{2-H}{2}}\d{u}.
	\end{equation*}
Hence, in order to extend the definition of the Rosenblatt process also for negative values of $t$, we define it via its increments as 
	\begin{equation*}
		R_t^H-R_s^H:=\int_{{\R}^2}^\prime\left(K(t,y_1,y_2)-K(s,y_1,y_2)\right)\d{W}_{y_1}\d{W}_{y_2}, \quad s,t\in\R
	\end{equation*}
and, in particular, one obtains $R_t^H$ by taking $s=0$ in the above definition. Let $n\in\mathbb{N}$ and $t_i, s_i\in\R$ such that $s_i<t_i$ for $i=1,2,\ldots, n$. Similarly as in the one-sided case, the distribution of the vector $(R^H_{t_1}-R^H_{s_1}, R^H_{t_2}-R^H_{s_2}, \ldots, R^H_{t_n}-R^H_{s_n})$ is determined by the distribution of the random variable 
	\begin{equation*}
		R:=\sum_{i=1}^n\theta_i(R^H_{t_i}-R^H_{s_i}) = \int_{{\R}^2}^\prime\left(\sum_{i=1}^n\theta_i(K(t_i,y_1,y_2)-K(s_i,y_1,y_2))\right)\d{W}_{y_1}\d{W}_{y_2}.
	\end{equation*}
Notice that the sum inside the integral is a symmetric function in the variables $y_1$ and $y_2$ and since this is a second order multiple integral, its distribution is determined by its cumulants. In particular, using formula (18) from \cite{Taqqu11}, we have that $\kappa_1(R) = 0$ and
	\begin{equation*}
		\kappa_k(R) = 2^{k-1}(k-1)!\sigma^k \sum_{r_1,\ldots, r_k\in\{1,\ldots,n\}}\theta_{r_1}\ldots\theta_{r_k}S((s_{r_1},t_{r_1}), \ldots, (s_{r_1},t_{r_1})), \quad k=2,3,\ldots
	\end{equation*}
where 
	\begin{align*}
		S((s_{r_1},t_{r_1}), \ldots, (s_{r_1},t_{r_1})) & := \\
		& \hspace{-4cm} := \int_{s_{r_1}}^{t_{r_1}}\int_{s_{r_2}}^{t_{r_2}}\cdots\int_{s_{r_k}}^{t_{r_k}}|x_1-x_2|^{H-1} |x_2-x_3|^{H-1}\cdots |x_{k-1}-x_k|^{H-1}|x_k-x_1|^{H-1} \d{x}_k\ldots\d{x}_1,
	\end{align*}
(cf. \cite[section 4, formulas (12), (13) and (15)]{Taqqu11}). The cumulants will not change if we consider $t_i+h$ and $s_i+h$ instead of $t_i$ and $s_i$, respectively and, therefore, the two-sided Rosenblatt process has stationary increments. Similar argument shows that its increments are reflexive as well. Moreover, it follows that for $-\infty <s_1<t_1<\infty$ and $-\infty < s_2<t_2<\infty$, we have that
	\begin{equation*}
		\E(R^H_{t_1}-R^H_{s_1})(R^H_{t_2}-R^H_{s_2}) = H(2H-1)\int_{s_1}^{t_1}\int_{s_2}^{t_2}|u-v|^{2H-2}\d{u}\d{v}
	\end{equation*}
and we see that the increments of the Rosenblatt process are correlated in the same way is the increments of the fBm (cf. formula \eqref{eq:cov_increments}). Hence, $R^H$ is also a two-sided Volterra process with the kernel $K^H$ given by \eqref{eq:fBm_kernel}. In particular, the covariance of $R^H$ is given by \eqref{eq:cov_fBm}. 
\end{example}

\begin{remark}
As in the case of the fBm, one should in fact think of $R_t^H$ as the difference $\tilde{R}^H_t-\tilde{R}^H_0$, where 
	\begin{equation*}
\tilde{R}^H_t:= c_H\int_{-\infty}^t\int_{-\infty}^t\left(\int_{y_1\vee y_2}^t(u-y_1)^{-\frac{2-H}{2}}(u-y_2)^{-\frac{2-H}{2}}\d{u}\right)\d{W}_{y_1}\d{W}_{y_2}.
	\end{equation*}
Of course, similarly as in the case of the fBm, this expression does not make sense, since the integrand is not square-integrable. 
\end{remark}

\subsection{Stochastic integration}
Let $(V, \langle\cdot,\cdot\rangle_V)$ be a separable Hilbert space. Let $b=(b_t, t\in \R)$ be a two-sided Volterra process with a kernel $K$. Denote by $\mathcal{E}(\R;V)$ the set of $V$-valued step functions on $\R$, i.e. $f\in\mathcal{E}(\R;V)$ satisfies
	\begin{equation*}
		f = \sum_{j=1}^{n} f_j\bm{1}_{[t_{j-1},t_j)}
	\end{equation*}
where $n\in\mathbb{N}$, $-\infty <t_0<t_1<\ldots <t_n<\infty$ and $f_j\in V$ for all $j=1,2,\ldots,n$. Note that we identify functions equal almost everywhere. Consider the linear mapping $i:\mathcal{E}(\R;V)\rightarrow L^2(\Omega;V)$ given by
	\begin{equation*}
		i:\quad f:=\sum_{j=1}^n f_i\bm{1}_{[t_{j-1},t_j)}\quad \longmapsto \quad\sum_{j}f_j(b_{t_{j}}-b_{t_{j-1}})=:i(f)
	\end{equation*}
and define the operator $\mathcal{K}^*: \mathcal{E}(\R;V)\rightarrow L^2(\R;V)$ by
	\begin{equation*}
		(\mathcal{K}^*f)(r):=\int_r^\infty f(u)\frac{\partial K}{\partial u}(u,r)\d{u}, \quad r\in \R.
	\end{equation*}
For simplicity, it is assumed here that $\mathcal{K}^*$ is injective. If this is not the case, then the quotient space $\tilde{\mathcal{E}}(\R;V):=\mathcal{E}(\R;V)/\ker\mathcal{K}^*$ may be considered after lifting $\mathcal{K}^*$ to $\tilde{\mathcal{E}}(\R;V)$ (cf. \cite{CouMas16}). Formula \eqref{eq:generalized_VP_increments} implies that 
	\begin{equation}
	\label{eq:ito_isometry}
		\|i(f)\|_{L^2(\Omega;V)} = \|\mathcal{K}^*f\|_{L^2(\R;V)}
	\end{equation}
for $f\in\mathcal{E}(\R;V)$. Now, have $\mathcal{E}(\R;V)$ completed under the inner product 
	\begin{equation}
	\label{eq:D_norm}
		\langle f,g\rangle_{\mathcal{D}} := \langle\mathcal{K}^*f,\mathcal{K}^*g\rangle_{L^2(\R,V)},
	\end{equation}
denote the completion by $\mathcal{D}(\R;V)$ and extend $\mathcal{K}^*$ to $(\mathcal{D}(\R;V),\langle\cdot,\cdot\rangle_{\mathcal{D}})$ which is now a Hilbert space. This in turn extends $i$ to a linear isometry between $\mathcal{D}(\R;V)$ and a closed linear subspace of $L^2(\Omega;V)$. The space $\mathcal{D}(\R;V)$ is viewed as the space of admissible integrands and, for $f\in\mathcal{D}(\R;V)$, the random variable $i(f)$ is the \textit{stochastic integral} of $f$ with respect to the Volterra process $b$. Whenever necessary we will also use the symbol
	\begin{equation*}
		\int f\d{b} := i(f).
	\end{equation*}

Since $\mathcal{D}(I;\R)$ can be a very large space, its elements might not be functions. The following lemma shows that the Lebesgue-Bochner space $L^\frac{2}{1+2\alpha}(\R;V)$ can be viewed as a subspace of $\mathcal{D}(\R;V)$.

\begin{lemma}
We have that the space $L^\frac{2}{1+2\alpha}(\R;V)$ is continuously embedded in $\mathcal{D}(\R;V)$.
\end{lemma}

\begin{proof}
Let $f\in\mathcal{E}(\R;V)$. Then
	\begin{equation*}
		\|f\|_{\mathcal{D}}=\|\mathcal{K}^*f\|_{L^2(\R;V)}\lesssim \|I_{-}^\alpha(f)\|_{L^2(\R;V)} \lesssim \|f\|_{L^\frac{2}{1+2\alpha}(\R;V)}
	\end{equation*}
by the Fubini theorem and \cite[Theorem 5.3]{SamKilMar93}. Here, $I^\alpha_{-}$ denotes the left-sided fractional integral on the real axis. The claim follows by standard approximation.
\end{proof}

\begin{remark}
Let $-\infty<s<t<\infty$. Then the definite integral $i_{s,t}(f)$ for $f\in L^\frac{2}{1+2\alpha}_{\mathrm{loc}}(\R;V)$ is defined as
	\begin{equation*}
		i_{s,t}(f) := \int_s^tf\d{b}:= i(\bm{1}_{[s,t)}f).
	\end{equation*}
\end{remark}

The following lemma will become useful in various calculations in the sequel.

\begin{lemma}
\label{lem:inner_prod} 
Set
	\begin{equation*}
		\phi(u,v) := \int_{-\infty}^{u\wedge v}\frac{\partial K}{\partial u}(u,r)\frac{\partial K}{\partial v}(v,r)\d{r}.
	\end{equation*}
The following claims hold:
	 \begin{enumerate}
	 	\item[(i)] For $u\neq v$, we have that 
					\begin{equation*}	 		
	 					\phi(u,v) \lesssim |u-v|^{2\alpha-1}.
	 				\end{equation*}
	 	\item[(ii)] Let $s_1<t_1$ and $s_2<t_2$. Let further $f\in L^\frac{2}{1+2\alpha}(s_1,t_1;V)$, $g\in L^\frac{2}{1+2\alpha}(s_2,t_2;V)$. Then 
					\begin{equation*}
						\langle i_{s_1,t_1}(f),i_{s_2,t_2}(g)\rangle_{L^2(\Omega; V)} = \int_{s_2}^{t_2}\int_{s_1}^{t_1}\langle f(u), g(v)\rangle_V\phi(u,v)\d u\d v.
					\end{equation*}
					In particular, we have that 
						\begin{equation*}
							\E(b_{t_1}-b_{s_1})(b_{t_2}-b_{s_2}) = \int_{s_1}^{t_1}\int_{s_2}^{t_2}\phi(u,v)\d u\d v.
						\end{equation*}
	 	\item[(iii)] Let $-\infty <s <t <\infty$ and $h\in V$ and $f\in L^\frac{2}{1+2\alpha}(s,t;V)$. Then 
	 				\begin{equation*}
	 					\left\langle h, \int_s^t f(r)\d{b}_{r}\right\rangle_V = \int_s^t\langle h, f(r)\rangle_V\d{b}_r.
	 				\end{equation*}
	 \end{enumerate}
\end{lemma}

\begin{proof}
The claim (i) follows by using \eqref{eq:kernel_estimate} and the substitution $z=\frac{v-r}{u-r}$ for $u<v$. Claims (ii) and (iii) follow by standard approximation arguments.
\end{proof}

The next proposition allows to work with the integral $\int_0^tf(r)\d{b}_r$ instead of the convolution integral $\int_0^tf(t-r)\d{b}_r$. Its multidimensional version is used in the proofs of \autoref{prop:limiting_measure} and \autoref{prop:strict_stationarity}.

\begin{proposition}
\label{lem:law_of_integrals}
Assume that $b$ has stationary and reflexive increments. Then for every function $f\in L^\frac{2}{1+2\alpha}_{\mathrm{loc}}([0,\infty);V)$ we have that 
	\begin{equation*}
	 \int_0^tf(t-r)\d b_r\overset{Law}{=} \int_0^tf(r)\d b_r \overset{Law}{=} \int_{-t}^0f(-u)\d{b}_u
	\end{equation*}
for every $t\geq 0$. 
\end{proposition}

\begin{proof}
We shall only prove the first equality. The second follows by similar arguments. Assume that $f$ is a simple function of the form
	\begin{equation*}
		f(r) = \sum_{i=1}^{m}f_i\bm{1}_{[t_{i-1},t_{i})}, \quad r\in [0,t],
	\end{equation*}
for some $n\in\mathbb{N}$, some partition $\{0=t_0<\ldots< t_n=t\}$ and $f_i\in V$, $i=1,2,\ldots,n$. The stochastic integrals take the form
	\begin{align*}
		\vec{\imath}_t(f) &:=\int_0^tf(r)\d b_r = \sum_{i=1}^mf_i\left(b_{t_i}-b_{t_{i-1}}\right),\\
		\cev{\imath}_t(f) &:= \int_0^tf(t-r)\d b_{r}  = \sum_{i=1}^mf_i\left(b_{t-t_{i-1}}-b_{t-t_{i}}\right).
	\end{align*}
By stationarity and reflexivity of the increments, we have that 
	\begin{align*}
		Law(b_{t-t_0}-b_{t-t_1}, b_{t-t_1}-b_{t-t_2}, \ldots, b_{t-t_{m-1}}-b_{t-t_m}) & = \\
		& \hspace{-5cm} = Law(b_{-t_0}-b_{-t_1}, b_{-t_1}-b_{-t_2}, \ldots, b_{-t_{m-1}}-b_{-t_m})\\
		& \hspace{-5cm} = Law(b_{t_1}-b_{t_0}, b_{t_2}-b_{t_1},\ldots, b_{t_m}-b_{t_{m-1}}).
	\end{align*}
Therefore, the probability laws of $\vec{\imath}_t(f)$ and $\cev{\imath}_t(f)$ must be equal. Now, let $f\in L^\frac{2}{1+2\alpha}(0,t;V)$ and let $\{f^{(n)}\}$ be a sequence of step functions such that $f^{n}\rightarrow f$ as $n\rightarrow\infty$ in $L^\frac{2}{1+2\alpha}(0,t;V)$. Clearly, $\vec{\imath_t}(f^n)\rightarrow \vec{\imath_t}(f)$ and $\cev{\imath_t}(f^n)\rightarrow \cev{\imath_t}(f)$ in $L^2(\Omega;V)$. Then we have that $\mu_{\vec{i}_t(f^{(n)})} = \mu_{\cev{i}_t(f^{(n)})}$ for each $n\in\mathbb{N}$ and thus,
	\begin{equation*}
		\mu_{\,\vec{i}_t(f)} =  \limweak_{n\rightarrow\infty} \mu_{\,\vec{i}_t(f^{n})} = \limweak_{n\rightarrow\infty} \mu_{\,\cev{i}_t(f^{n})} = \mu_{\,\cev{i}_t(f)}
	\end{equation*}
where $\mu_Y$ denotes the probability law of the random variable $Y$.
\end{proof}

In order to consider stochastic evolution equations, a Volterra process with values in a Hilbert space must be introduced. 

\begin{definition}
\label{def:cylindrical_process}
Let $U$ be a real separable Hilbert space. \textit{$U$-cylindrical $\alpha$-regular Volterra process} is a collection $B=(B_t, t\in\R)$ of bounded linear operators $B_t:U\rightarrow L^2(\Omega)$ such that 
\begin{itemize}
\itemsep-0.5em
	\item for every $u\in U$, $B(u)$ is a centered stochastic process in $\R$ with $B_0(u)=0$;
	\item for every $s_1,t_1,s_2,t_2\in \R$ and every $u_1,u_2\in U$ it holds that
		\begin{equation}
		\label{eq:correlation_of_cylindrical}
			\E\left(B_{t_1}(u_1)-B_{s_1}(u_2)\right)\left(B_{t_2}(u_2)-B_{s_2}(u_2)\right)  = R(s_1,t_1,s_2,t_2)\langle u_1,u_2\rangle_V
		\end{equation}
with $R$ given by \eqref{eq:definition_of_R}.
\end{itemize}
\end{definition}

\begin{remark}
If $B$ is a $U$-cylindrical $\alpha$-regular Volterra process, then for every complete orthonormal basis $\{e_n\}$ of $U$ there is a sequence $\{b^{(n)}\}$ of uncorrelated scalar $\alpha$-regular Volterra processes such that for every $u\in U$, we have
	\begin{equation}
	\label{eq:representation_of_cylindrical}
		B_t(u) = \sum_{n}\langle u,e_n\rangle_Ub^{(n)}_t.
	\end{equation}
In fact, the sequence $\{b^{(n)}\}$ is given by $b^{(n)}=B(e_n)$. By uncorrelated, we mean that 
	\begin{equation*}
		\E(b_{t_1}^{(n)}-b_{s_1}^{(n)})(b_{t_2}^{(m)}-b_{s_2}^{(m)}) = 0, \quad m\neq n,
	\end{equation*}
for every $t_1,s_1,t_2,s_2\in \R$ which clearly holds by \eqref{eq:correlation_of_cylindrical}. Note that although each $b^{(n)}$ might be a different process, they have the same kernel (e.g. if $U={\R}^2$, then $b^{(1)}$ might be the fBm of $H>\frac{1}{2}$ and $b^{(2)}$ the Rosenblatt process of the same $H$). On the other hand, given an orthonormal basis $\{e_n\}$ of $U$ and a sequence of uncorrelated $\alpha$-regular Volterra processes $\{b^{(n)}\}$, the sum \eqref{eq:representation_of_cylindrical} defines a $U$-cylindrical $\alpha$-regular Volterra process.
\end{remark}

\begin{definition}
Let $U$ be a real separable Hilbert space and $B$ be a $U$-cylindrical $\alpha$-regular Volterra process. We say that $B$ has \textit{stationary (or reflexive) increments} if for every $n$ and every $u_1,u_2,\ldots, u_n\in U$, the process $\bm{b}=(B(u_1), B(u_2), \ldots, B(u_n))$ has stationary (or reflexive) increments in the sense of \autoref{def:stationarity_reflexivity}.
\end{definition}	

An important case of $U$-cylindrical Volterra processes are the Gaussian ones.

\begin{definition}
We say that a $U$-cylindrical $\alpha$-regular Volterra process $B$ is \textit{Gaussian} if for every $u_1, u_2\in U$ and every $s,t\in\R$, the random vector $(B_s(u_1),B_t(u_2))$ is jointly Gaussian.
\end{definition}

\begin{remark}
Note that if $B$ is a $U$-cylindrical $\alpha$-regular Volterra process which is Gaussian, then for every orthonormal basis $\{e_n\}$ of $U$, the sequence $\{b^{(n)}\}=\{B(e_n)\}$ consists of mutually independent processes.
\end{remark}

An integral of operator-valued functions with respect to a $U$-cylindrical $\alpha$-regular Volterra process is further defined. The following construction is similar to the one given in \cite[Section 3]{CouMas16} in the case of one-sided $U$-cylindrical $\alpha$-regular Volterra processes.

\begin{definition}
An operator $G\in U\rightarrow \mathcal{D}(\R;V)$ is called \textit{elementary}, if 
	\begin{equation*}
		Gu = \sum_{k=1}^Kg_k\langle u,e_k\rangle_U
	\end{equation*}
holds for every $u\in U$ where $\{g_k\}\subset\mathcal{D}(\R;V)$ and $\{e_n\}$ is a complete orthonormal basis of $U$.
\end{definition}

Let $G\in\L(U,\mathcal{D}(\R;V))$, $B$ be a cylindrical $\alpha$-regular Volterra process and $\{e_k\}$ a complete orthonormal basis of $U$. Let $I(G)$ be the (elementary) integral
	\begin{equation*}
		I(G):=\sum_{k=1}^K \int g_k\d{b}^{(k)}
	\end{equation*}
where $b^{(k)}=B(e_k)$. As usual, we have to extend the operator $I$ to a larger space of operators. Since $b^{(k)}$ are uncorrelated, we obtain
	\begin{equation*}
		\|I(G)\|_{L^2(\Omega;V)} ^2 = \E\left|\sum_{k=1}^K\int g_k\d{b}^{(k)}\right|_V^2 = \sum_{k=1}^K\E\left|\int g_k\d{b}^{(k)}\right|_V^2 = \sum_{k=1}^K\|Ge_k\|^2_{\mathcal{D}(\R;V)}.
	\end{equation*}
In other words, we have that 
	\begin{equation}
	\label{eq:Ito_isome_infinite}
		\|I(G)\|_{L^2(\Omega)} = \|G\|_{\L_2(U,\mathcal{D}(\R;V))}.
	\end{equation}
for elementary operators $G$. Using formula \eqref{eq:Ito_isome_infinite}, we may extend the operator $I$ to a linear isometry between the space $\L_2(U,\mathcal{D}(\R;V))$ and a closed linear subspace of $L^2(\Omega;V)$.

\begin{definition}
A bounded operator $G: U\rightarrow \mathcal{D}(\R;V)$ is called \textit{stochastically integrable} with respect to a $U$-cylindrical  $\alpha$-regular Volterra process $B$ if $G\in\mathcal{L}_2(U,\mathcal{D}(\R;V))$. In this case, $I(G)$ is called a \textit{stochastic integral} of $G$ with respect to $B$.
\end{definition}

Naturally, an operator $G\in\L_2(U,\mathcal{D}(\R;V))$ may be identified with a deterministic operator-valued map $G:\R\rightarrow\L_2(U,V)$ and we do so in the sequel. The following proposition will allow us to define the definite stochastic integral with respect to $B$.

\begin{proposition}
\label{prop:integrable_operator}
Let $G\in L^\frac{2}{1+2\alpha}_{\mathrm{loc}}(\R;\L_2(U,V))$ and let $-\infty<s<t<\infty$. Then the function $\bm{1}_{[s,t)}G$ is stochastically integrable with respect to a $U$-cylindrical $\alpha$-regular Volterra process $B$. 
\end{proposition}

\begin{proof}
Using \eqref{eq:D_norm}, (ii) and (i) of \autoref{lem:inner_prod}, and the H\"{o}lder and Hardy-Littlewood inequalities successively, we obtain
	\begin{align*}
		\|\bm{1}_{[s,t)}G\|_{\L_2(U,\mathcal{D}(\R;V))}^2 & = \sum_{n}\left\|\mathcal{K}^*\bm{1}_{[s,t)}Ge_n\right\|_{L^2(\R;V)}^2\\
		& = \int_s^t\int_s^t \langle G(u),G(v)\rangle_{\L_2(U,V)}\phi(u,v)\d{u}\d{v}\\
		& \lesssim \int_s^t\int_s^t \|G(u)\|_{\L_2(U,V)}\|G(v)\|_{\L_2(U,V)}|u-v|^{2\alpha-1}\d{u}\d{v}\\
		& \lesssim \left(\int_s^t\|G(u)\|_{\L_2(U,V)}^\frac{2}{1+2\alpha}\d{u}\right)^{1+2\alpha}<\infty.
	\end{align*}
\end{proof}

It follows from \autoref{prop:integrable_operator} that for $G\in L^\frac{2}{1+2\alpha}_{\mathrm{loc}}(\R;\L_2(U,V))$ we may set 
\begin{equation*}
	I_{s,t}(G):=\int_s^tG(r)\d{B}_r := I(\bm{1}_{[s,t)}G)
\end{equation*}
and call it the \textit{(definite) stochastic integral} of $G$ with respect to $B$ on $[s,t)$. In this case, it holds that
	\begin{equation}
	\label{eq:norm_of_integral}
		\|I_{s,t}(G)\|_{L^2(\Omega;V)}^2 = \int_s^t\int_s^t\langle G(u),G(v)\rangle_{\L_2(U,V)}\phi(u,v)\d{u}\d{v}.
	\end{equation}
Moreover, linearity of $I$ and the identity $\bm{1}_{[s,t)}=\bm{1}_{[s,x)} +\bm{1}_{[x,t)}$ ensure that $I_{s,t}(G) = I_{s,x}(G)+I_{x,t}(G)$ for every $s<x<t$.

\section{Limiting measure and stationary solutions}
\label{sec:limiting_measure}

Let $U,V$ be two real separable Hilbert spaces and consider the stochastic evolution equation 
	\begin{equation}
	\label{eq:SDE}
		\left\{\begin{array}{rl}
			\d X_t & = AX_t +\varPhi\d B_t, \quad t\geq 0,\\
				X_0 & = x
		\end{array}\right.
	\end{equation}
where $A$ is an infinitesimal generator of a strongly continuous semigroup $(S(t), t\geq 0)$ of bounded linear operators acting on $V$ and $x\in L^2(\Omega;V)$. We assume that $\varPhi\in\mathcal{L}(U,V)$ and $B=(B_t, t\in\R)$ is a $U$-cylindrical $\alpha$-regular Volterra process. The solution to \eqref{eq:SDE} is given in the mild form by the variation of constants formula
	\begin{equation}
	\label{eq:mild_solution_Zt}
		X_t^x := S(t)x+Z_t:= S(t)x +\int_0^tS(t-r)\varPhi\d B_r \quad t\geq 0.
	\end{equation}
	
Consider the following:
	\begin{enumerate}
	\item[\namedlabel{hyp:H}{\textbf{(H)}}] Let $S(r)\varPhi\in\L_2(U,V)$ for all $r>0$. Let there further exist $T_0>0$ such that 
			\begin{equation}
			\label{eq:H_integral}
				\int_0^{T_0} \|S(r)\varPhi\|_{\L_2(U,V)}^\frac{2}{1+2\alpha}\d{r} <\infty.
			\end{equation}
	\end{enumerate}

\begin{proposition}
\label{prop:existence}
If \ref{hyp:H} holds, then then the mild solution $X^x$ given by \eqref{eq:mild_solution_Zt} is a well-defined $V$-valued process which is mean-square right continuous and, in particular, it has a version with measurable sample paths.
\end{proposition}

\begin{proof}
The proof essentially follows the lines of the proof of Proposition 4.2 in \cite{CouMas16} for the assumption \textbf{(A3)} where $f\equiv 0$. Parts (i) and (ii) of \autoref{lem:inner_prod} are the two properties needed. In particular, first it is shown that \eqref{eq:H_integral} combined with the fact that $S(u)=S(u-T_0)S(T_0)$ for $u>T_0$ imply that 
	\begin{equation*}
		\int_0^t\|S(r)\varPhi\|_{\L_2(U,V)}^\frac{2}{1+2\alpha}\d{r} <\infty
	\end{equation*}
for every $t>0$ which assures that $Z_t$ is a well-defined $V$-valued random variable for every $t>0$ by \autoref{prop:integrable_operator}. For $0<s<t$, using additivity of the definite stochastic integral, we can write
	\begin{equation*}
		\E|Z_t-Z_s|_V^2 \lesssim \E\left|\int_s^t S(t-r)\varPhi\d{B}_r\right|_V^2 + \E\left|\int_0^s[S(t-s)-I]S(s-r)\varPhi\d{B}_r\right|_V^2
	\end{equation*}
and \eqref{eq:H_integral} together with the strong continuity of the semigroup $(S(t),t\geq 0)$ imply that both terms above tend to zero as $t\searrow s$. 
\end{proof}

\begin{proposition}
\label{prop:covariance_operator}
Assume that \ref{hyp:H} holds and let $T>0$. Then $Z=(Z_t, t\in [0,T])$ is an $L^2(0,T;V)$-valued random variable whose covariance operator $Q_T: L^2(0,T;V)\rightarrow L^2(0,T;V)$ takes the form
		\begin{equation*}
			\left(Q_T\varphi\right)(r) = \int_0^Tg(r,s)\varphi(s)\d s,
		\end{equation*}
with 
	\begin{equation*}
		g(r,s) := \int_0^r\int_0^sS(r-u)\varPhi\varPhi^*S^*(s-v)\phi(u,v)\d u\d v
	\end{equation*}
and $\phi$ is defined in \autoref{lem:inner_prod}.
\end{proposition}

\begin{proof}
Note first that $Z$ is an $L^2(0,T;V)$-valued random variable since by the proof of \autoref{prop:integrable_operator} we have that 
	\begin{equation*}
		\sup_{t\in [0,T]}\|Z_t\|_{L^2(\Omega;V)}\lesssim \sup_{t\in [0,T]}\left(\int_0^t\|S(r)\varPhi\|_{\L_2(U,V)}^\frac{2}{1+2\alpha}\d{r}\right)^{\alpha+\frac{1}{2}}
	\end{equation*}
which is finite by \eqref{eq:H_integral} similarly as in the proof of \autoref{prop:existence}. Let $\varphi, \psi\in L^2(0,T;V)$. Then
			\begin{align*}
				\langle Q_T\varphi,\psi\rangle_{L^2(0,T;V)} & = \E\int_0^T\langle\varphi(s),Z_s\rangle_V\d s\int_0^T\langle\psi(r),Z_r\rangle_V\d r\\
					 & = \int_0^T\int_0^T \E\langle\varphi(s),Z_s\rangle_V\langle\psi(r),Z_r\rangle_V\d r\d s.
			\end{align*}
Now, using the fact that $b^{(k)}$ and $b^{(l)}$ are uncorrelated and (iii) of \autoref{lem:inner_prod}, we obtain
			\begin{align*}
				\hspace{1cm} \E\langle\varphi(s),Z_s\rangle_V\langle\psi(r),Z_r\rangle_V & =\\
					& \hspace{-3cm} = \sum_{n,m}\E\left\langle\varphi(s), \int_0^sS(s-v)\varPhi e_n\d b_v^{(n)}\right\rangle_V\left\langle\psi(r),\int_0^rS(r-u)\varPhi e_m\d b_u^{(m)}\right\rangle_V\\
					& \hspace{-3cm} = \sum_{n} \E\int_0^s\langle\varphi(s),S(s-u)\varPhi e_n\rangle_V\d b_v^{(n)}\cdot\int_0^r\langle\psi(r),S(r-u)\varPhi e_n\rangle_V\d b_u^{(n)}.
			\end{align*}
Using (ii) of \autoref{lem:inner_prod}, we obtain
	\begin{align*}
				\hspace{1cm} \E\langle\varphi(s),Z_s\rangle_V\langle\psi(r),Z_r\rangle_V & =\\
					& \hspace{-3cm} = \sum_{n} \int_0^s\int_0^r\langle \varphi(s), S(s-v)\varPhi e_n\rangle_V\langle \psi(r), S(r-u)\varPhi e_n\rangle_V \phi(u,v)\d u\d v\\
					& \hspace{-3cm} = \sum_{n} \int_0^s\int_0^r\langle\varPhi^*S^*(s-v)\varphi(s),e_n\rangle_V\langle\varPhi^*S^*(r-u)\psi(r),e_n\rangle_V \phi(u,v)\d u\d v\\
					& \hspace{-3cm} =  \int_0^s\int_0^r\langle\varPhi^*S^*(s-v)\varphi(s), \varPhi^*S^*(r-u)\psi(r)\rangle_V\phi(u,v)\d u\d v\\
					& \hspace{-3cm} = \int_0^s\int_0^r \langle S(r-u)\varPhi\varPhi^*S^*(s-v)\varphi (s),\psi(r)\rangle_V\phi(u,v)\d u\d v.
			\end{align*}
The interchange of the sum and integrals is possible due to \eqref{eq:H_integral}. Hence
	\begin{align*}
		\hspace{1cm} \langle Q_T\varphi,\psi\rangle_{L^2(0,T;V)} & = \\
			& \hspace{-2cm} = \int_0^T\left\langle\int_0^T\left(\int_0^r\int_0^sS(r-u)\varPhi\varPhi^*S^*(s-v)\phi(u,v)\d u\d v\right)\varphi(s)\d s, \psi (r)\right\rangle_V\d r.
	\end{align*}
\end{proof}

We denote by $\mu_t^x$ the probability law $X_t^x$ in the rest of the paper. Clearly, the $V$-valued random variable $Z_t$ has the covariance operator 
	\begin{equation*}
		q_t = \int_0^t\int_0^tS(t-u)\varPhi\varPhi^*S^*(t-v)\phi(u,v)\d u\d v.
	\end{equation*}
and by \ref{hyp:H}, we have $\mathrm{Tr}\,q_t<\infty$. The following proposition will be useful in \autoref{ex:non_exp_stable}.

\begin{proposition}
\label{prop:necessary_condition}
Assume that \ref{hyp:H} holds and let $B$ be $U$-cylindrical $\alpha$-regular Gaussian Volterra process. If there is a measure $\mu_\infty$ such that $\limweak_{t\rightarrow\infty}\mu_t^0 = \mu_\infty$, then \[\sup_{t\geq 0}\mathrm{Tr}\,q_t<\infty.\] 
\end{proposition}

\begin{proof}
The proof is similar to the proof of \cite[Theorem 6.2.1, (i) $\Rightarrow$ (ii)]{DaPratoZabczykErgodicity}. If $B$ is Gaussian, then $\mu_t^0 = Law(Z_t) = N(0,q_t)$. Assume that there is a measure $\mu_\infty$ such that $\limweak_{t\rightarrow\infty}\mu_t^0=\mu_\infty$. Since the characteristic functional of $\mu_t^0$, $\varphi_{\mu_t^0}$, is given by 
	\begin{equation*}
		\varphi_{\mu_t^0}(h) = \exp(-1/2\langle q_th,h\rangle_V),
	\end{equation*}
we see that the characteristic functional of $\mu_{\infty}$, $\varphi_{\mu_\infty}$, must be real-valued because we have that $\varphi_{\mu_t^0}(h)\rightarrow \varphi_{\mu_\infty}(h)$ as $t\rightarrow\infty$ for all $h\in V$. According to Bochner-Minlos-Sazanov theorem (see e.g. \cite[Theorem 2.13]{DaPratoZab92} or \cite[Theorem 1.1.5]{ManRha04}), there exists a positive symmetric trace-class operator $O_{1/2}: V\rightarrow V$ such that for every $h\in V$ for which $\langle O_{1/2}h,h\rangle_V< 1$, we have that $\varphi_{\mu_\infty}$ satisfies $\varphi_{\mu_\infty}(h)> 1/2$. It follows from the convergence of $\varphi_{\mu_t^0}(h)$ to $\varphi_{\mu_\infty}(h)$ that for $h\in V$ such that $\langle O_{1/2}h,h\rangle_V<1$ we have that $\varphi_{\mu_t^0}(h)>1/2$ for sufficiently large $t$. Since we have that $\langle q_t h,h\rangle_V = 2\log\frac{1}{\varphi_{\mu_t^0}(h)}$, we obtain the implication
	\begin{equation*}
		h\in V: \langle O_{1/2}h,h\rangle_V<1 \quad \implies\quad \langle q_th,h\rangle_V < 2\log 2
	\end{equation*}
for sufficiently large $t$. It follows that $0\leq q_t < 2\log 2O_{1/2}$ and $0\leq \mathrm{Tr}\,q_t \leq 2\log 2\, \mathrm{Tr}\,O_{1/2}$ for sufficiently large $t$ which yields the claim.
\end{proof}

We now discuss the existence of a limiting measure of the process $Z$. We obtain an analogue of Proposition 3.4 from \cite{DunMasDun02}. Note however, that exponential stability of the semigroup is not assumed here.
	
\begin{proposition}
\label{prop:limiting_measure}
Assume that $B$ has stationary and reflexive increments. Assume further that $S(u)\varPhi\in\L_2(U,V)$ for every $u>0$ and that 
	\begin{equation}
	\label{eq:condition_on_inv_measure}
		\int_0^\infty \|S(r)\varPhi\|_{\L_2(U,V)}^\frac{2}{1+2\alpha}\d{r}<\infty
	\end{equation}
holds. Then there is a measure $\mu_\infty$ such that 
	\begin{equation*}
		\limweak_{t\rightarrow\infty}\mu_t^0 = \mu_\infty.
	\end{equation*}
\end{proposition}

\begin{proof} 
Define
	\begin{equation*}
		Z'_t:=\int_0^tS(u)\varPhi\d B_r, \quad t\geq 0.
	\end{equation*}
A similar approximation procedure as in the proof of Lemma \ref{lem:law_of_integrals} applies and hence, we have
	\begin{equation*}
			Law(Z'_t)=Law(Z_t)=\mu_t^0, \quad t\geq 0.
	\end{equation*}	
We claim that there is $Z'_\infty\in L^2(\Omega;V)$ such that $Z'_t\rightarrow Z'_\infty$ in $L^2(\Omega;V)$ as $t\rightarrow\infty$. Indeed, let $\{k_n\}\subset [0,\infty)$ such that $k_n\rightarrow\infty$ as $n\rightarrow\infty$. We will show that $\{Z'_{k_n}\}_{n\in\mathbb{N}}$ is Cauchy in $L^2(\Omega;V)$. Let $n,m\in\mathbb{N}, n>m$. Using \eqref{eq:norm_of_integral}, we have that
		\begin{align*}
			\E|Z'_{k_n}-Z'_{k_m}|_V^2 & = \E\left|\int_{k_m}^{k_n}S(u)\varPhi\d B_r\right|^2_V\\
				& = \int_{k_m}^{k_n}\int_{k_m}^{k_n} \langle S(u)\varPhi,S(v)\varPhi\rangle_{\L_2(U,V)}\phi(u,v)\d u\d v\\
				& \lesssim \left(\int_{k_m}^{k_n}\|S(u)\varPhi\|_{\L_2(U,V)}^\frac{2}{1+2\alpha}\d{r}\right)^{1+2\alpha}\\
				& \leq \left(\int_{k_m}^\infty\|S(r)\varPhi\|_{\L_2(U,V)}^\frac{2}{1+2\alpha}\d{r}\right)^{1+2\alpha}
		\end{align*}
Now, if we let $m\rightarrow\infty$, the last integral tends to zero by \eqref{eq:condition_on_inv_measure}. Hence, $\{Z'_{k_n}\}_{n\in\mathbb{N}}$ is Cauchy in $L^2(\Omega;V)$ and there must be a limit $Z_\infty^k$. Let $\{l_n\}_{n\in\mathbb{N}}\subset[0,\infty)$ be another sequence such that $l_n\rightarrow\infty$ as $n\rightarrow\infty$ and let $Z_\infty^l$ be the corresponding limit constructed as above. We then have
	\begin{equation*}
		\|Z_\infty^k-Z_\infty^l\|_{L^2(\Omega;V)} \leq \|Z_\infty^k-Z_{k_n}'\|_{L^2(\Omega;V)} + \|Z_{k_n}'-Z_{l_n}'\|_{L^2(\Omega;V)} + \|Z_{l_n}'-Z_\infty^l\|_{L^2(\Omega;V)}
	\end{equation*}
and similarly as before, it can be shown that the middle term tends to zero as $n\rightarrow \infty$ and hence, $Z_\infty^k=Z_\infty^l=:Z'_\infty$.
\end{proof}

If the semigroup $(S(t), t\geq 0)$ is strongly stable, then $\mu_\infty$ is a limiting measure for the solution $X^x$ for every initial condition $x\in L^2(\Omega;V)$.

\begin{proposition}
\label{prop:limiting_measure_2}
Let the assumptions of \autoref{prop:limiting_measure} hold. Let $x\in L^2(\Omega;V)$ be such that  
	\begin{equation}
	\label{eq:stability_of_initial_condition}
			\lim_{t\rightarrow\infty}|S(t)x|_V=0
	\end{equation}
almost surely. Then 
	\begin{equation}
	\label{eq:weak_limit}
		\limweak_{t\rightarrow\infty}\mu_t^x = \mu_\infty.
	\end{equation}
In particular, if the semigroup $(S(t),t\geq 0)$ is strongly stable (i.e. $S(t)\rightarrow 0$ as $t\rightarrow \infty$ in the strong operator topology), then \eqref{eq:weak_limit} holds for every $x\in L^2(\Omega;V)$. 
\end{proposition}

\begin{proof}
Let $\mu_\infty$ be the law of $Z_{\infty}'$ from the proof of \autoref{prop:limiting_measure} and let $g:V\rightarrow\mathbb{R}$ be a bounded Lipschitz continuous functional. Then we have 
	\begin{align*}
		\left|\int_Vg\d\mu_t^x-\int_Vg\d\mu_\infty\right| &=\left|\E g\left(S(t)x+Z'_t\right) - \E g\left(Z'_\infty\right)\right|_V\\
		& \leq \E\left|g\left(S(t)x+Z'_t\right)-g\left(Z'_\infty\right)\right|_V\\
		& \leq C_{\mathrm{Lip}}\left(\E|S(t)x|_V+\E\left|Z'_t-Z'_\infty\right|_V\right)
	\end{align*}
which tends to zero as $t\rightarrow\infty$ by the proof of \autoref{prop:limiting_measure}.
\end{proof}

It is well-known that if the equation \eqref{eq:SDE} is driven by the cylindrical Wiener process, the measure $\mu_\infty$ is invariant. As noted in \cite{DunMasDun02}, this fails to be true when the driving process is the cylindrical fBm $B^H$. Indeed, if $Law(x)=\mu_\infty$ and if $x$ is independent of the process $B^H$, then $\mu_t^x$ might not remain constant. However, as the next proposition, this may be achieved for one particular initial condition. The proposition is an analogous to \cite[Theorem 3.1]{MasPos08}.

\begin{proposition}
\label{prop:strict_stationarity}
Let the assumptions of \autoref{prop:limiting_measure} hold. Then there is a $V$-valued random variable $x_\infty$ such that $(X^{x_\infty}_t, t\geq 0)$ is a strictly stationary process such that $Law(X^{x_\infty}_t) = \mu_\infty$ for all $t\geq 0$.
\end{proposition}

\begin{proof}
Denote
	\begin{equation*}
		Z_t'' := \int_{-t}^0S(-u)\varPhi\d B_u.
	\end{equation*}
In a similar manner as in \autoref{lem:law_of_integrals}, it can be inferred that 
	\begin{equation*}
		Law(Z''_t) = Law (Z_t) = \mu_t^0, \quad t\geq 0, 
	\end{equation*}
and following the lines of the proof of \autoref{prop:limiting_measure}, we can show that there is $V$-valued random variable $x_\infty$ such that $Z''_t\rightarrow x_\infty$ in $L^2(\Omega;V)$ as $t\rightarrow\infty$. Clearly, the probability law of $x_\infty$ is $\mu_\infty$. Now, let $t, h\geq 0$. Then we have that 
	\begin{align*}
		X_{t+h}^{x_\infty}& = S(t+h)x_\infty+Z_{t+h}\\
			& = \lim_{n\rightarrow\infty}\left(\int_{-n}^0S(t+h-u)\varPhi\d B_u+\int_0^{t+h}S(t+h-u)\varPhi\d B_u\right)\\
			& = \lim_{n\rightarrow\infty}\int_{-n}^{t+h}S(t+h-u)\varPhi\d B_u.
	\end{align*}
Notice that 
	\begin{equation*}
		X_{t+h}^{x_\infty} = \lim_{n\rightarrow\infty}\int_{-n-h}^{t}S(t-v)\varPhi\d B_{v+h} = \lim_{n\rightarrow\infty}\int_{-n}^tS(t-v)\varPhi\d B_{v+h},
	\end{equation*}
since if $n\rightarrow\infty$, then $h+n\rightarrow\infty$ for all $h\geq 0$. Let $k\in\mathbb{N}$ and $t_1, \ldots, t_k\geq 0$ be arbitrary times. Next we show that
	\begin{equation*}
		Law\left(X_{t_1}^{x_\infty}, \ldots, X_{t_k}^{x_\infty}\right) = Law\left(X_{t_1+h}^{x_\infty}, \ldots, X_{t_k+h}^{x_\infty}\right)
	\end{equation*}
for all $h\geq 0$. Consider
	\begin{align*}
		Law\left(X_{t_1+h}^{x_\infty}, \ldots, X_{t_k+h}^{x_\infty}\right) & = \\
			& \hspace{-2cm} = \limweak_{n\rightarrow\infty} Law\left(\int_{-n}^{t_1+h}S(t_1+h-u)\varPhi\d B_u, \ldots, \int_{-n}^{t_k+h}S(t_k+h-u)\varPhi\d B_u\right)\\
			& \hspace{-2cm} = \limweak_{n\rightarrow\infty} Law\left(\int_{-n-h}^{t_1}S(t_1-v)\varPhi\d B_{v+h}, \ldots, \int_{-n-h}^{t_k}S(t_k-v)\varPhi\d B_{v+h}\right)\\
			& \hspace{-2cm} = \limweak_{n\rightarrow\infty} Law\left(\int_{-n}^{t_1}S(t_1-v)\varPhi\d B_{v+h}, \ldots, \int_{-n}^{t_k}S(t_k-v)\varPhi\d B_{v+h}\right)\\
			& \hspace{-2cm} = \limweak_{n\rightarrow\infty} Law\left(\int_{-n}^{t_1}S(t_1-v)\varPhi\d B_{v}, \ldots, \int_{-n}^{t_k}S(t_k-v)\varPhi\d B_{v}\right)\\
			& \hspace{-2cm} = Law\left(X_{t_1}^{x_\infty}, \ldots, X_{t_k}^{x_\infty}\right).
	\end{align*}
The fourth equality follows from the fact that $B$ has stationary increments.
\end{proof}

\section{Examples}
\label{sec:examples}

\begin{example}
\label{ex:non_exp_stable}
Let $U=\R$, $V=L^2(0,\infty)$ and $A$ be the infinitesimal generator of the left-shift semigroup $(S(t),t\geq 0)$, i.e. $A:\Dom(A):=W^{1,2}(0,\infty)\rightarrow L^2(0,\infty)$ is given by $Af:=f^\prime$ and generates the $C_0$-semigroup $(S(t),t\geq 0)$ on $L^2(0,\infty)$ given by 
	\begin{equation*}
		S(t)g(\xi)=g(t+\xi)
	\end{equation*}
for $g\in L^2(0,\infty)$, $t\geq 0$ and $\xi\in[0,\infty)$. Let $W^H=(W^H_t, t\in\R)$ be a scalar fBm with a fixed $H>1/2$ and let $\varphi:[0,\infty)\rightarrow\R$ be given by $\varphi(\xi):= (\xi+1)^{-\beta}$ with some $\beta>1/2$. Let $\varPhi_{\varphi}\in\L(\R,L^2(0,\infty))$ be given by $\varPhi_\varphi(c)(\xi):= c\varphi(\xi)$ for $c\in\R$. Consider the following equation:
	\begin{equation}
	\label{eq:SDE_phi}
		\d{X}_t = AX_t\d{t} + \varPhi_\varphi\d{W}^H_t, \quad t\geq 0.
	\end{equation}
Denote the solution to \eqref{eq:SDE_phi} with $X_0=x\in L^2(\Omega;L^2(0,\infty))$ by $X^x_\varphi$.

If $\beta>H+1/2$, we have that 
	\begin{equation*}
		\int_0^\infty\left(\int_0^\infty(\xi+r+1)^{-2\beta}\d{\xi}\right)^\frac{1}{2H}\d{r} <\infty.
	\end{equation*}
Hence, by \autoref{prop:limiting_measure}, that there is a limiting measure $\mu_\infty$ for $X^0_\varphi$. Furthermore, by \autoref{prop:strict_stationarity}, there is a $x_\infty\in L^2(\Omega;L^2(0,\infty))$ such that the solution $X^{x_\infty}_\varphi$ is strictly stationary. Notice that for every $z\in L^2(0,\infty)$ the function $\zeta_z: [0,\infty)\rightarrow [0,\infty)$ defined by 
	\begin{equation*}
		\zeta_z(t):= \left(\int_0^\infty|z(t+\xi)|^2\d{\xi}\right)^\frac{1}{2}
	\end{equation*}
is non-increasing (since the integrand is non-negative) and it holds that $\lim_{t\rightarrow\infty}\zeta_z(t)=0$ (i.e. the semigroup $(S(t),t\geq 0)$ is strongly stable). Hence, by \autoref{prop:limiting_measure_2}, we have that for every initial condition $x\in L^2(\Omega;L^2(0,\infty))$, the limiting measure for $X^x_\varphi$ to \eqref{eq:SDE_phi} exists and equals $\mu_\infty$. Moreover, we have 
	\begin{equation*}
		\sup_{t\geq 0}\mathrm{Tr}\,q_t = H(2H-1) \int_0^\infty\int_0^\infty \left(\int_0^\infty(u+\xi+1)^{-\beta}(v+\xi+1)^{-\beta}\d{\xi}\right)|u-v|^{2H-2}\d{u}\d{v} <\infty.
	\end{equation*}

Let us look closer at the integral $J(\beta,H)$:
	\begin{align*}
 J(\beta,H) & := \int_0^\infty\int_0^\infty\left(\int_0^\infty(u+\xi+1)^{-\beta}(v+\xi+1)^{-\beta}\d{\xi}\right)|u-v|^{2H-2}\d{u}\d{v} \\
		& = 2\int_0^\infty\int_1^\infty (u+\xi)^{-\beta}\left(\int_0^u(v+\xi)^{-\beta}(u-v)^{2H-2}\d{v}\right)\d{\xi}\d{u}
	\end{align*}
using the Tonelli theorem. Since
	\begin{equation*}
		\int_0^u(v+\xi)^{-\beta}(u-v)^{2H-2}\d{v} = \xi^{-\beta}u^{2H-1}\int_0^1\left(1+z\frac{u}{\xi}\right)^{-\beta}(1-z)^{2H-2}\d{z},
	\end{equation*}
we have 
	\begin{align*}
		J(\beta,H) & = \\
		& \hspace{-11mm} = 2\int_1^\infty \xi^{2H-2\beta-1}\left(\int_0^\infty \left(1+\frac{u}{\xi}\right)^{-\beta}\left(\frac{u}{\xi}\right)^{2H-1}\left(\int_0^1\left(1+z\frac{u}{\xi}\right)^{-\beta}(1-z)^{2H-2}\d{z}\right)\d{u}\right)\d{\xi}\\
		& \hspace{-11mm} = 2\left[\int_1^\infty \xi^{2H-2\beta}\d{\xi}\right]\left[\int_0^1(1-z)^{2H-2}\left(\int_0^\infty (1+y)^{-\beta}y^{2H-1}(1+zy)^{-\beta}\d{y}\right)\d{z}\right].\numberthis\label{eq:chain}\\
\intertext{Clearly, $J(\beta,H)= \infty$ if $\beta\leq H+1/2$. If $\beta>H+1/2$, the first integral is finite and we can continue the chain \eqref{eq:chain} as}
		& \hspace{-11mm} = \frac{2\Gamma(2H)\Gamma(2\beta-2H)}{\Gamma(2\beta)}\left(\int_1^\infty \xi^{2H-2\beta}\d{\xi}\right)\left(\int_0^1(1-z)^{2H-2}{_2F_1}(\beta,2H,2\beta,1-z)\d{z}\right)
	\end{align*}
where $\Gamma$ is the Gamma function and $_2F_1$ is the Gauss hypergeometric function. The last equality follows by formula (1.6.7) on p. 20 in \cite{Slat66} which can be used if $\beta>H$. The second integral can be written as 
	\begin{equation*}
		\int_0^1 z^{2H-2}{_2F_1}(\beta,2H,2\beta,z)\d{z}.
	\end{equation*}
If $\beta>2H$, $_2F_1(\beta,2H,2\beta,z)$ is finite at $z=1$ (see \cite[Formula (15.4.20) on p. 387]{NIST10}). If $\beta=2H$, it holds that 
	\begin{equation*}
		\lim_{z\rightarrow 1-}\frac{_2F_1(\beta,2H, 2\beta,z)}{-\log(1-z)} = \frac{\Gamma(2H+\beta)}{\Gamma(2H)\Gamma(\beta)}.
	\end{equation*}
(see \cite[Formula (15.4.21) on p. 387]{NIST10}) and if $\beta<2H$, it holds that 
\begin{equation*}
		\lim_{z\rightarrow 1-}\frac{_2F_1(\beta, 2H,2\beta,z)}{(1-z)^{\beta-2H}}=\frac{\Gamma(2\beta)\Gamma(2H-\beta)}{\Gamma(\beta)\Gamma(2H)}
	\end{equation*}	
(see \cite[Formula (15.4.23) on p. 387]{NIST10}). Therefore, if $\beta>H$, the integral converges. Hence, we have that $J(\beta,H)<\infty$ if and only if $\beta > H+1/2$ and it follows that
	\begin{equation*}
		J(\beta,H) <\infty \quad\mbox{ if and only if }\quad \int_0^\infty \left(\int_0^\infty|\varphi(\xi+r)|^2\d{\xi}\right)^\frac{1}{2H}\d{r} <\infty.
	\end{equation*}
Hence, if $\beta\leq H+1/2$, we have that $J(\beta,H)=\infty$ and consequently, $\sup_{t\geq 0}\mathrm{Tr}\,q_t=\infty$. By \autoref{prop:necessary_condition}, there cannot be a limiting measure for the solution $X^\varphi$. Now, recall that for this equation driven by the one-dimensional Wiener process, the necessary and sufficient condition for the existence of a limiting measure is
	\begin{equation*}
		\int_0^\infty \int_0^\infty |\varphi(\xi+r)|^2\d{\xi}\d{r} <\infty
	\end{equation*}
(see e.g. \cite[Example 11.9]{DaPratoZab92}). Hence, if we choose $\beta\in (1, H+1/2]$, we obtain, that the solution \eqref{eq:SDE_phi} driven by the Wiener process does admit a limiting measure, whereas if the driving noise is the fBm of $H>1/2$, it does not.
\end{example}

\begin{example}
Consider the stochastic heat equation on an open bounded domain $\O\subset\R^d$ with $\mathcal{C}^1$ boundary perturbed by an additive space-time noise $\eta$ which is Rosenblatt (with a given parameter $H\in (1/2,1)$) in time and white in space, i.e.
	\begin{equation*}
		\partial_t u = \Delta u + \eta
	\end{equation*}
on ${\R}_+\times\O$ with the Dirichlet boundary condition $u=0$ on ${\R}_+\times\partial\O$. This problem can be formulated in terms of the stochastic evolution equation \eqref{eq:SDE}. Define \[\Dom(A):= W^{2,2}(\O)\cap W^{1,2}_0(\O)\] and take $A:=\Delta|_{\Dom(A)}$, $V:=L^2(\O)$. Let $\{e_n\}$ be a complete orthonormal basis of $U:=L^2(\O)$ and let $\{R^{(n)}\}$ be a sequence of independent copies of the two-sided Rosenblatt process $R^H$ with a fixed $H\in (1/2,1)$. Define 
	\begin{equation*}
		B_t(g) := \sum_{n=1}^\infty\langle g, e_n\rangle_{L^2(\O)}R^{(n)}_t
	\end{equation*}
for $g\in L^2(\O)$ and $t\in\R$. Then $B=(B_t, t\in\R)$ is an $L^2(\O)$-cylindrical $\alpha$-regular Volterra process ($\alpha:= H-1/2$) with stationary and reflexive increments. Assume that $\varPhi\in\L(L^2(\O))$. Formally, we can write that
	\begin{equation*}
		\eta(t,\xi) = [\varPhi\dot{B}_t](\xi)
	\end{equation*}
for $(t,\xi)\in{\R}_+\times\O$. The operator $A$ generates a strongly continuous, exponentially stable semigroup $(S(t), t\geq 0)$ on the space $L^2(\O)$ and standard estimates on its Green kernel yield
	\begin{equation*}
		\|S(r)\varPhi\|_{\L_2(L^2(\O))} \lesssim r^{-\frac{d}{4}}
	\end{equation*}
for $r>0$ (see \cite[\S 3, Theorem 2]{Arima64}). Hence, if $d<4H$, the convolution integral 
	\begin{equation*}
		Z_t = \int_0^tS(t-r)\varPhi\d{B}_r, \quad t\geq 0,
	\end{equation*}
is well-defined and has values in $L^2(\O)$ since in this case, \ref{hyp:H} holds. In particular, for $d=1,2$, there is no restriction on $H\in (1/2,1)$ and if $d=3$, the parameter $H$ has to be greater than $3/4$. Since the Rosenblatt process has stationary and reflexive increments and since we assume that $\{R^{(n)}\}$ are independent, the process $B$ has stationary and reflexive increments. Moreover, we can write 
	\begin{align*}
		\int_0^\infty \|S(u)\varPhi\|_{\L_2(L^2(\O))}^\frac{2}{1+2\alpha}\d{r} & =  \int_0^{T_0}\|S(u)\varPhi\|_{\L_2(L^2(\O))}^\frac{2}{1+2\alpha}\d{r} +\\ & \hspace{2cm}+\int_{T_0}^\infty \|S(u-T_0)\|_{\L(L^2(\O))}^\frac{2}{1+2\alpha}\|S(T_0)\varPhi\|_{\L_2(L^2(\O))}^\frac{2}{1+2\alpha}\d{r}
	\end{align*}
which is finite by \ref{hyp:H} and exponential stability of $(S(t), t\geq 0)$. Exponential stability of the semigroup $(S(t), t\geq 0)$ implies its strong stability and hence, we may appeal to \autoref{prop:limiting_measure_2} and infer that for each $x\in L^2(\Omega;L^2(\O))$, the solution 
	\begin{equation*}
		X^x_t=S(t)x+Z_t, \quad t\geq 0,
	\end{equation*}
admits a limiting measure $\mu_\infty$. Moreover, there is a random variable $x_\infty$ (whose distribution is $\mu_\infty$) such that the solution $X^{x_\infty}$ is a strictly stationary process by \autoref{prop:strict_stationarity}.
\end{example}

ACKNOWLEDGEMENT: I wish to thank Bohdan Maslowski and Martin Ondrej\'{a}t for our fruitful discussions and for their valuable suggestions.

\bibliographystyle{siamplain}
\bibliography{references}

\end{document}